\documentclass{proc-l-hijacked}
\usepackage{amssymb, amsmath, amsthm, hyperref,}

\newtheorem{theorem}{Theorem}[section]

\newtheorem{corollary}[theorem]{Corollary}

\theoremstyle{definition}

\newtheorem{remark}[theorem]{Remarks}

\DeclareMathOperator{\re}{Re}

\DeclareMathOperator{\pco}{pco}

\numberwithin{equation}{section}

\begin{document}

	\title[]{Applications of the Growth Characteristics Induced by the Spectral Distance}
	\author{R. Brits}
	\address{Department of Mathematics, University of Johannesburg, South Africa}
	\email{rbrits@uj.ac.za }

\begin{abstract}
Let $A$ be a complex unital Banach algebra. Using a connection between the spectral distance and the growth characteristics of a certain entire map into $A$, we derive a generalization of Gelfand's famous Power Boundedness Theorem. Elaborating on these ideas, with the help of a Phragm\'{e}n-Lindel\"{o}f device for subharmonic functions, it is then shown, as the main result, that two normal elements of a $C^*$-algebra are equal if and only if they are quasinilpotent equivalent.
\end{abstract}
	\parindent 0mm
	
	\maketitle

\section{Introduction}
Asymptotically intertwined operators on Banach spaces, and the associated notions of the spectral distance and quasinilpotent equivalence, are  well-established topics in the context of local spectral theory. Standard references (by now) include \cite{c+f, f+v, l+n, vas3, vas1, vas2}, and some very recent ones \cite{duggal, duggal2}. The aim of this paper is to deviate a little from this approach, and to see what can be obtained from a more global perspective on these matters. It is therefore natural to revert to the situation of a complex unital Banach algebra. Throughout we shall rely on connections with the classical theory of growth of entire functions (which can be suitably extended to Banach algebra valued functions via the separation properties of the Hahn Banach Theorem). This approach, though somewhat unexplored in general, was also considered in \cite{frunza} by Frunz\v a who obtained some very neat results for the case of Hilbert space operators. A generalization of asymptotically intertwined operators, the spectral distance, and quasinilpotent equivalence to abstract Banach algebras (\cite{razpet}) is quite obvious: Let $A$ denote a complex Banach algebra with identity $\mathbf 1$. For $a,b\in A$ associate operators $L_a$, $R_b$, and $C_{a,b}$, acting on $A$, by
the relations
\[L_ax=ax, \quad R_bx=xb,\quad \hbox{and}\quad C_{a,b}x=(L_a-R_b)x\quad (x\in A).\]
Since $L_a$ and $R_b$ commute it is easy that
\[C_{a,b}^nx=\sum_{k=0}^n(-1)^k{n \choose k}a^{n-k}xb^k\quad (x\in A),\] with the convention that, if $0\not=a\in A$,
then $a^0=\mathbf 1$.
Using the particular value $x=\mathbf 1$, define $\rho:A\times A\rightarrow \mathbb R$ by
\begin{equation}\label{varrho}
\rho (a,b)=\limsup_n\left\| C_{a,b}^n\mathbf 1\right\| ^{1/n},
\end{equation}
and then define
\begin{equation}\label{rho}
d_\rho(a,b)=\sup \{\rho (a,b),\rho (b,a)\}.
\end{equation}
The function $d_\rho$, which defines a semimetric on $A$, is a noncommutative generalization of the semimetric induced by the spectral radius in the commutative case. That is, $d_\rho(a,b)=\lim_n\|(a-b)^n\|^{1/n}$ provided $ab=ba$.
If $X$ is a Banach space, and $S,T\in A:=\mathcal L(X)$, the Banach algebra of bounded linear operators from $X$ into $X$, then the number $\rho(S,T)$ is a well-known quantity called the \emph{local spectral radius} \cite[p.235]{l+n} of the commutator $C_{S,T}\in\mathcal L(A)$ at $I$. The number $d_\rho(S,T)$ is called the \emph{spectral distance} \cite[p.251]{l+n} of the operators $S$ and $T$. Furthermore, the (ordered) pair $(S,T)$ is said to be \emph{asymptotically intertwined} \cite[p.248]{l+n} by the identity, $I$, if $\rho(S,T)=0$. If each of the pairs $(S,T)$ and $(T,S)$ is asymptotically intertwined by the identity operator (i.e. $d_\rho(S,T)=0$), then $S$ and $T$ are called \emph{quasinilpotent equivalent} \cite[p.253]{l+n}.  We shall adopt the operator terminology for $a,b$ in a general Banach algebra $A$ even though $d_\rho$ is (more accurately) called the \emph{spectral semidistance} in \cite{razpet} (see also the comments preceding Proposition 3.4.9 in \cite{l+n}). Since one may always embed $A$ isometrically into $\mathcal L(X)$ for some Banach space $X$, it follows from \cite[Proposition 3.4.11]{l+n} that $d_\rho(a,b)=0$ implies equality of the spectra of $a$ and $b$. This will not be particularly useful to us, since most of our results rely only on the weaker  $\rho(a,b)=0$. The following very simple facts are useful for this paper: $\rho(\alpha a,\alpha b)=|\alpha|\rho(a,b)$ for $\alpha\in\mathbb C$; $\rho(a,b)=\rho(a+c,b+c)$ for any $c\in A$ commuting with both $b$ and $a$; any two quasinilpotent elements are quasinilpotent equivalent.

The methods employed in this paper stem from a curious connection between $\rho$ and the growth characteristics of a certain entire map from $\mathbb C$ into $A$. We recall the following definitions: Let $f$ be an entire function from $\mathbb C$ into $A$. Then $f$ has an everywhere convergent power series expansion
\[f(\lambda)=\sum_{n=0}^\infty a_n\lambda^n,\] with coefficients $a_n$ belonging to $A$. Next define a function \[M_f(r)=\sup\limits_{|\lambda|\leq r}\|f(\lambda)\|\quad (0<r\in\mathbb R).\] The function $f$ is said to be of \emph{finite order} if there exist $K>0$ and $R>0$ such that $M_f(r)<e^{r^K}$ holds for all $r>R$. The infimum of the set of positive real numbers, $K$, such that the preceding inequality holds is called the \emph{order} of $f$, denoted by $\omega_f$. If $\omega_f=1$ then $f$ is said to be of \emph{exponential order}. Suppose $f$ is entire, and of finite order $\omega:=\omega_f$. Then $f$ is said to be of \emph{finite type} if there exist $L>0$ and $R>0$ such that $M_f(r)<e^{Lr^{\omega}}$ holds for all $r>R$. The infimum of the set of positive real numbers, $L$, such that the preceding inequality holds is called the \emph{type} of $f$, denoted by $\tau_f$.
\begin{remark}
	We should point out that the terminology ``exponential order" is somewhat nonstandard; the phrase ``exponential type", which is more commonly used, usually signifies (i) $\omega_f<1$ and $\tau_f\leq\infty$, or (ii) $\omega_f=1$ and $\tau_f<\infty$.
\end{remark}
It is known (see the monograph \cite[p.41]{levin}) that the order and type of an entire function are given by the respective formulas
\[\omega_f=\limsup_n\left(\frac{n\log n}{\log\|a_n\|^{-1}}\right)\mbox{ and }\tau_f=\frac{1}{e\omega_f}\limsup_n\left(n\sqrt[n]{{\|a_n\|^{\omega_f}}}\right).\]
If, for example, $f$ is of order $0$ and finite type, then it follows from the definition (via Liouville's Theorem) that $f$ must be constant; but there are other, more interesting, scenarios involving order and type which may also force this (cf. Theorem~\ref{POL}).
Let $a,b\in A$, and define
\begin{equation}\label{efun}
f:\lambda\mapsto e^{\lambda a}e^{-\lambda b},\quad (\lambda\in\mathbb C).
\end{equation}
The corresponding series expansion, valid for all $\lambda\in\mathbb C$, is given by
\begin{equation}\label{FE}
f(\lambda)=e^{{\lambda}a}e^{-{\lambda}b}=\sum_{n=0}^\infty\frac{{\lambda}^nC_{a,b}^n\mathbf 1}{n!}.
\end{equation}
Since $\|f(\lambda)\|\leq e^{\left(\|a\|+\|b\|\right)|\lambda|}$, for all $\lambda\in\mathbb C$, it is immediate, from the definition, that $\omega_f\leq1$. Suppose we know that $f$ is of exponential order (i.e. $\omega_f=1$). Using Stirling's formula it follows that $\lim_nn(1/n!)^{1/n}=e$, from which we subsequently obtain
\[\tau_f=\frac{1}{e}\limsup_n\left(n(1/n!)^{1/n}\left\| C_{a,b}^n\mathbf 1\right\| ^{1/n}\right)=\rho(a,b).\]
In conclusion, the function $f$ defined in \eqref{efun} exhibits the following growth characteristics: either $w_f<1$, or $w_f=1$ in which case $\tau_f=\rho(a,b)$;
this observation will be crucial throughout the remainder of this paper. Furthermore we shall adopt the following notation: If $A^{-1}$ denotes the invertible group of $A$, and $x\in A$ then  $\sigma_A(x):=\{\lambda\in\mathbb C:\lambda\mathbf 1-x\not\in A^{-1}\}$ denotes the spectrum of $x$, and $r_{\sigma_A}(x):=\sup\{|\lambda|:\lambda\in\sigma_A(x)\}=\lim_n\|x^n\|^{1/n}$ the spectral radius of $x$. If the algebra under discussion is clear from the context, we may omit the subscript $A$ in the aforementioned notation. If $K$ is a compact subset of $\mathbb C$, then $\pco(K)$ is the polynomially convex hull of $K$.

\section{ Quasinilpotent equivalence and power boundedness}

We start with a useful result  which implies that quasinilpotent equivalence is preserved under the Holomorphic Functional Calculus. Theorem~\ref{FCT} extends \cite[Proposition 4]{ frunza}.
\begin{theorem}\label{FCT} Let $A$ be a Banach algebra, $a,b\in A$, and let $f$ be holomorphic on an open set containing $\sigma(a)\cup\sigma(b)$. Then
	\begin{equation}\label{FCE}
	\rho(f(a),f(b))\leq\sup_{\omega\in\sigma(b)}\left\{\sup_{\substack{\lambda\in\sigma(a)\\ \lambda\not=\omega}}\left|\frac{f(\lambda)-f(\omega)}{\lambda-\omega}\right|,|f^\prime(\omega)| \right\}\cdot\rho(a,b)
	\end{equation}
	In particular if $d_\rho(a,b)=0$ (so that $\sigma(a)=\sigma(b)$), and $f$ is analytic on an open set $U$ containing $\sigma(a)$, then $d_\rho(f(a),f(b))=0$.
\end{theorem}
\begin{proof} Let $\Gamma$ be a contour in $U$ surrounding $\sigma(a)\cup\sigma(b)$ (in the sense of \cite[10.23, p.259--260]{rudin1}). Since the map $x\mapsto L_x$
	isomorphically embeds $A$ onto a closed subalgebra of $\mathcal L(A)$ it follows that
	\[f(L_a)=\frac{1}{2\pi\,i}\int\limits_\Gamma f(\lambda)(\lambda I-L_a)^{-1}d\lambda=L_{f(a)}.\]
	Similarly
	\[f(R_b)=\frac{1}{2\pi\,i}\int\limits_\Gamma f(\lambda)(\lambda I-R_b)^{-1}d\lambda=R_{f(b)}.\]
	Now, because $L_a$ and $R_b$ commute, we have
	\begin{align*}
	L_{f(a)}-R_{f(b)}&=\frac{1}{2\pi\,i}\int\limits_\Gamma f(\lambda)
	\left[(\lambda I-L_a)^{-1}-(\lambda I-R_b)^{-1}\right]d\lambda\\&=
	\frac{1}{2\pi\,i}\int\limits_\Gamma f(\lambda)
	\left[(\lambda I-L_a)^{-1}(\lambda I-R_b)^{-1}(L_a-R_b)\right]d\lambda\\&=
	\left(\frac{1}{2\pi\,i}\int\limits_\Gamma f(\lambda)
	\left[(\lambda I-L_a)^{-1}(\lambda I-R_b)^{-1}\right]d\lambda\right)\left(L_a-R_b\right)\\&=
	\left(L_a-R_b\right)\left(\frac{1}{2\pi\,i}\int\limits_\Gamma f(\lambda)
	\left[(\lambda I-L_a)^{-1}(\lambda I-R_b)^{-1}\right]d\lambda\right).
	\end{align*}
	The preceding argument shows that $L_{f(a)}-R_{f(b)}$ can factored as
	$L_{f(a)}-R_{f(b)}=H(L_a,R_b)[L_a-R_b]$ where  $H(L_a,R_b)\in\mathcal L(A)$ commutes with $L_a$ and $R_b$.
	Thus
	\begin{align*}
	\rho(f(a),f(b))&=\limsup_n\|(L_{f(a)}-R_{f(b)})^n\mathbf 1\|^{1/n}\\&=
	\limsup_n\|H^n(L_a,R_b)[L_a-R_b]^n\mathbf 1\|^{1/n}\\&\leq
	\limsup_n\|H^n(L_a,R_b)\|^{1/n}\,\|[L_a-R_b]^n\mathbf 1\|^{1/n}\\&=
	r_{\sigma}\left(H(L_a,R_b)\right)\rho(a,b).
	\end{align*}
	Let $C^{\prime}$ be the bicommutant of $\{L_a,R_b\}\subset\mathcal L(A)$ and let $\chi$ be a character on $C^{\prime}$. Then
	\begin{align*}
	\chi(H(L_a,R_b))&=\frac{1}{2\pi\,i}\int\limits_\Gamma \frac{f(\lambda)}
	{(\lambda-\chi(L_a))(\lambda-\chi(R_b))}d\lambda\\&=
	\left\{
	\begin{array}{cl}
	f^\prime\left(\chi(R_b)\right) & \text{if } \chi(R_b)=\chi(L_a)\\
	\frac{f(\chi(L_a))-f(\chi(R_b))}{\chi(L_a)-\chi(R_b)} & \text{if } \chi(R_b)\not=\chi(L_a).\\
	\end{array} \right.
	\end{align*}
	Since $\sigma(a)=\sigma_{\mathcal L(A)}(L_a)=\sigma_{C^{\prime}}(L_a)$ and
	$\sigma(b)=\sigma_{\mathcal L(A)}(R_b)=\sigma_{C^{\prime}}(R_b)$ \eqref{FCE} is evident, and all that remains is to show that the right side of \eqref{FCE} is finite:
	Define $G:U\times U\rightarrow\mathbb C$ by
	\[G(\omega,\lambda)=\left\{
	\begin{array}{cl}
	\frac{f(\lambda)-f(\omega)}{\lambda-\omega} & \text{if } \lambda\in U-\{\omega\}\\
	f^\prime(\omega) & \text{if } \lambda=\omega.\\
	\end{array} \right.\]
	It follows directly from \cite[Lemma 10.29]{rudin2} that $G$ is continuous on $U\times U$, and hence the right side of \eqref{FCE} is finite since 
	$\sigma(a)$ and $\sigma(b)$ are compact.
\end{proof}
As an easy application of Theorem~\ref{FCT} we give a very short proof of the main result in \cite{razpet}. First make the observation that if $p$ and $q$ are idempotents satisfying $\rho(p,q)=0$, then $p=q$; this follows directly from the fact that for $n$ odd we have $\left\| C_{p,q}^n\mathbf 1\right\|=\|p-q\|$.
\begin{theorem}[Razpet]\label{RAZ} Suppose $d_\rho(a,b)=0$, and suppose $h$ is an entire complex function with simple zeros. If the corresponding holomorphic calculus elements $h(a)$ and $h(b)$ satisfy $h(a)=h(b)=0$, then $a=b$.
\end{theorem}
\begin{proof}
	As stated in \cite{razpet} it is well-known that $\sigma(a)=\sigma(b)$ is finite, and, in particular, the hypothesis on $h$ implies that $a$ and $b$ can be expressed respectively as
	\[a=\sum_{j=1}^n\lambda_jp_j\ \ \mbox{ and }\ \ b=\sum_{j=1}^n\lambda_jq_j\] where $\{\lambda_1,\dots,\lambda_n\}=\sigma(a)=\sigma(b)$, and $p_j$ and $q_j$ are the Riesz idempotents corresponding to $\lambda_j$ for $a$ and $b$ respectively. Let $\{B(\lambda_j,r)\}_{j=1}^n$ be a covering of $\{\lambda_1,\dots,\lambda_n\}$ by mutually disjoint open disks, and for $k$ arbitrary but fixed let $f_k$ be the function that takes the value $1$ on $B(\lambda_k,r)$ and the value $0$ on $B(\lambda_j,r)$ if $j\not=k$. It follows from Theorem~\ref{FCT} that
	\[d_\rho(p_k,q_k)=d_\rho(f_k(a),f_k(b))=0\] whence $p_k=q_k$ by the comments preceding Theorem~\ref{RAZ}. Since $k$ was arbitrary $a=b$.
\end{proof}
In a more general (than Theorem~\ref{RAZ}) but similar vein, suppose now that $d_\rho(a,b)=0$ and assume $\sigma(a)$ is finite. So we can write $\sigma(a)=\{\lambda_1,\dots,\lambda_n\}=\sigma(b)$. The holomorphic functional calculus implies the representation
\[a=\sum_{j=1}^n\lambda_jp_j+r_a\ \ \mbox{ and }\ \ b=\sum_{j=1}^n\lambda_jq_j+r_b\]where, for each $j$, $p_j$ and $q_j$ are the Riesz idempotents corresponding to $\lambda_j$ for $a$ and $b$ respectively, and $r_a$  and $r_b$ are quasinilpotent elements belonging to the bicommutant of $a$ and $b$ respectively. The exact same argument as in Theorem~\ref{RAZ} shows that $p_j=q_j$ for each $j$. However, the observation preceding Theorem~\ref{RAZ} is no longer valid if idempotents are replaced by quasinilpotents; the best possible result is thus that $a-b=r_a-r_b$. This is Theorem 4.3 in  \cite{b+r}. Other results related to quasinilpotent equivalence of elements with finite spectra in \cite{b+r} can just as easily be derived through the use of Theorem~\ref{FCT}.  An application of Theorem~\ref{FCT}, together with the relationship between $\tau_f$ and $\rho(a,b)$, leads to a natural generalization of a famous result due to Gelfand:  Let $A$ be a complex Banach algebra with identity. Then $a\in A$ is the identity element of $A$ if and only if $\sigma(a)=\{1\}$ and $\{\|a^n\|:n\in\mathbb Z\}$ is bounded (i.e. $a$ is so-called doubly power bounded). Gelfand's result seems (deceptively) quite simple but it really has a very rich history, and, moreover, even the shortest of its many arguments relies on sophisticated ideas. For a particularly elegant proof, due to Allan and Ransford, look at \cite[Theorem 2.6.12.]{palmer}. To prove Theorem~\ref{GG} we shall need an old theorem ascribed to P\'{o}lya:
\begin{theorem}[\cite{levinson}]\label{POL} Let $f:\mathbb C\rightarrow A$ be an entire function from the field $\mathbb C$ into a Banach algebra $A$. If $f$ is (norm) bounded over $\mathbb Z$ and if $$\limsup\limits_{r\rightarrow\infty}\frac{\log M_f(r)}{r}\leq0,$$ then $f$ is constant.
\end{theorem}
In the original P\'{o}lya's Theorem $A=\mathbb C$, but it extends easily to $A$ via a standard argument with functionals belonging to the dual of $A$.
\begin{theorem}\label{GG}
	Let $A$ be a Banach algebra, and let $a,b\in A$. If $0\notin\pco(\sigma(a)\cup\sigma(b))$, then $a=b$ if and only if $\rho(a,b)=0$ and $\sup_{n\in\mathbb Z}\|a^nb^{-n}\|<+\infty.$ More generally, two elements, $a$ and $b$, in a Banach algebra coincide if and only if $\rho(a,b)=0$ , and there exists  $\alpha\notin\pco(\sigma(a)\cup\sigma(b))$ such that
	$\sup_{n\in\mathbb Z}\|(\alpha\mathbf 1+a)^n(\alpha\mathbf 1+b)^{-n}\|<+\infty.$
\end{theorem}
\begin{proof} The forward implication is trivial. To prove the reverse implication, suppose there exists $M>0$ such that $\|a^nb^{-n}\|\leq M$ for all $n\in\mathbb Z$.   Since $\sigma(a)\cup\sigma(b)$ does not separate $0$ from infinity, there is an open set $U$ containing $\sigma(a)\cup\sigma(b)$ such that $\lambda=e^{\log\lambda}$ for some analytic branch, $\log\lambda$, of the logarithm on $U$. Since $\rho(a,b)=0$ it follows, from Theorem~\ref{FCT}, that
	$\rho(\log a,\log b)=0$. So we consider the function from $\mathbb C$ to $A$ given by
	\[h(\lambda)=e^{\lambda\log a}e^{-\lambda\log b}=\sum_{n=0}^\infty\frac{{\lambda}^nC_{\log a,\log b}^n\mathbf 1}{n!}.\]
	As noted in the preceding section, the function $h$ is of order at most one. We consider two cases:\\
	(i) $\omega_h<1$: This implies there exists $0<K<1$ and $R>0$ such that $M_h(r)<e^{r^K}$ holds for all $r>R$ whence it follows that
	$$\limsup\limits_{r\rightarrow\infty}\frac{\log M_h(r)}{r}\leq\limsup\limits_{r\rightarrow\infty}\frac{1}{r^{1-K}}=0.$$\\
	(ii) $\omega_h=1$: From the discussion in the preceding section we now have that $\tau_h=\rho(\log a,\log b)=0$. If $\epsilon>0$ is given, then there exists $R(\epsilon)>0$ such that
	$M_h(r)<e^{\epsilon r}$ holds for all $r>R(\epsilon)$ whence it follows that
	$$\limsup\limits_{r\rightarrow\infty}\frac{\log M_h(r)}{r}\leq\limsup\limits_{r\rightarrow\infty}\frac{\log e^{\epsilon r}}{r}=\epsilon.$$ Since $\epsilon>0$ was arbitrary we also have in this case that
	$$\limsup\limits_{r\rightarrow\infty}\frac{\log M_h(r)}{r}\leq0.$$
	If we observe further that \[\|h(n)\|=\|e^{n\log a}e^{-n\log b}\|=\|a^{n}b^{-n}\|\leq M\]
	holds for each $n\in\mathbb Z$, then it follows from  P\'{o}lya's Theorem that $h$ must be constant. From this one infers that $\log a=\log b$, and hence that $a=b$.
	For the general case we simplify have to observe that $\rho(a,b)=0\Leftrightarrow\rho(\alpha\mathbf 1+a,\alpha\mathbf 1+b)=0$ for all $\alpha\in\mathbb C$.
\end{proof}

\begin{corollary}[Gelfand]\label{GEL}
	Let $A$ be a Banach algebra and $a\in A$. Then $a=\mathbf 1$ if and only if $\sigma(a)=\{1\}$ and $\sup_{n\in\mathbb Z}\|a^n\|<+\infty.$
\end{corollary}

\begin{proof}
	The forward implication is trivial. Conversely, if $\sigma(a)=\{1\}$, then $0\notin\pco(\sigma(a))$ and $\rho(a,\mathbf 1)=r_\sigma(a-\mathbf 1)=0$. If, in addition,
	$\{\|a^n\|:n\in\mathbb Z\}$ is bounded then Theorem~\ref{GG} says $a=\mathbf 1$.
\end{proof}

\section{ Quasinilpotent equivalence in $C^*$-algebras}

Two immediate applications of Theorem~\ref{GG} to $C^*$-algebras are:

\begin{corollary}\label{SAE}
	Let $A$ be a $C^*$-algebra, and let $a,b\in A$ be self-adjoint elements. Then $a=b$ if and only if $\rho(a,b)=0$.
\end{corollary}

\begin{proof} Assume $\rho(a,b)=0$. Let $r>0$ be sufficiently small so that $\sigma(ta)\cup\sigma(tb)\subseteq[-\frac{\pi}{2},\frac{\pi}{2}]$ for all real $t\in[-r,r]$. Then observe that
	\[\sigma(e^{ita})\cup\sigma(e^{itb})\subseteq\{\lambda\in\mathbb C:|\lambda|=1,\re\lambda\geq0\},\ \ \ t\in[-r,r].\] Furthermore, by Theorem~\ref{FCT}, we have that
	$\rho(e^{ita},e^{itb})=0$. Since $0\notin\pco(\sigma(e^{ita})\cup\sigma(e^{itb}))$, and $$\|e^{nita}e^{-nitb}\|\leq\|e^{nita}\|\,\|e^{-nitb}\|=1$$ holds for all $n\in\mathbb Z$, Theorem~\ref{GG} gives that $e^{ita}=e^{itb}$. But this being valid for all $t\in[-r,r]$ we obtain $a=b$.
	
\end{proof}

\begin{corollary}\label{UE}
	Let $A$ be a $C^*$-algebra and let $a,b\in A$ be unitary elements. Then $a=b$ if and only if $\rho(a,b)=0$.
\end{corollary}

\begin{proof} Assume $\rho(a,b)=0$. If $\sigma(a)\cup\sigma(b)$ is properly contained in the unit circle, $S$, then $0\notin\pco(\sigma(a)\cup\sigma(b))$, and for each $n\in\mathbb Z$
	\[\|a^nb^{-n}\|\leq\|a^n\|\|b^{-n}\|=r_\sigma(a^n)\,r_\sigma(b^{-n})=1.\] So, Theorem~\ref{GG} gives $a=b$. Assume that
	$\sigma(a)\cup\sigma(b)=S$. Write
	\begin{equation}\label{ab}
	a=\frac{1}{2}\left(a+a^{*}\right)+\frac{1}{2i}(a-a^{*})\,i,\mbox{ and }
	b=\frac{1}{2}\left(b+b^{*}\right)+\frac{1}{2i}(b-b^{*})\,i.
	\end{equation}
	It follows from Theorem~\ref{FCT}, with
	$$f(\lambda)=\lambda+\frac{1}{\lambda},\ \ \lambda\in U:=\{\lambda\in\mathbb C:\frac{1}{2}<|\lambda|<2\},$$
	that
	$$\rho(a+a^*, b+b^*)=\rho(a+a^{-1},b+b^{-1})=\rho(f(a),f(b))=0,$$ and hence, from Theorem~\ref{SAE}, that
	$a+a^*=b+b^*$. If, on the other hand, we take $f(\lambda)=\lambda-\frac{1}{\lambda}$ on $U$ then similarly it follows
	that $\rho(a-a^{-1},b-b^{-1})=0$. This implies that  $\rho({-i}(a-a^*),{-i}(b-b^*)=0$, and again using Theorem~\ref{SAE} that $a-a^*=b-b^*$.
	Equality of $a$ and $b$ is now clear from \eqref{ab}.
\end{proof}

We show that Corollaries~\ref{SAE} and ~\ref{UE} can in fact be extended to normal elements. The proof now relies on a Phragm\'{e}n-Lindel\"{o}f Theorem for subharmonic functions:

\begin{theorem}[\cite{rans} Corollary 2.3.8]\label{PHT}
	Let $u$ be a subharmonic function on the half-plane $H:=\{\lambda\in\mathbb C:\re\lambda>0\}$, such that for some constants $A,B<\infty$
	\begin{equation}\label{AB}
	u(\lambda)\leq A+B|\lambda|\quad\mbox{ for all }\lambda\in H.
	\end{equation}
	If
	\begin{equation}\label{BDY}
	\limsup\limits_{\lambda\rightarrow\zeta}u(\lambda)\leq0\quad\mbox{ for all } \zeta\in\partial H\backslash\{\infty\},
	\end{equation}
	and if
	\begin{equation}\label{IRA}
	\limsup\limits_{{t\rightarrow\infty}\atop{t\in\mathbb R^+}}\frac{u(t)}{t}=L,
	\end{equation}
	then
	\begin{equation}
	u(\lambda)\leq L\re\lambda\quad\mbox{ for all }\lambda\in H.
	\end{equation}
\end{theorem}

\begin{theorem}\label{NORMAL}
	Let $A$ be a $C^*$-algebra and let $a,b\in A$ be normal elements. Then $a=b$ if and only if $\rho(a,b)=0$.
\end{theorem}

\begin{proof}
	We first notice, since
	\begin{equation}\label{invol}
	\left[C_{a,b}^n\mathbf 1\right]^*=\left\{\begin{array}{cc}C_{b^*,a^*}^n\mathbf 1 & n\mbox{ even }\\
	-C_{b^*,a^*}^n\mathbf 1 & n\mbox{ odd }\end{array}\right.,
	\end{equation}
	that $\rho(a,b)=0\Rightarrow\rho(b^*,a^*)=0.$
	Define an entire function, $f$, and an entire (auxiliary) function, $g$, from $\mathbb C$ into $A$ by respectively
	$$f(\lambda)=e^{\lambda i a }e^{-\lambda i (b-b^*) }e^{-\lambda i a^* }\mbox{ and }
	g(\lambda)=e^{\lambda i (a-a^*) }e^{-\lambda i (b-b^*) }.$$
	Since $a$ is normal it follows, from Jacobson's Lemma \cite[Lemma 3.1.2]{aup}, that
	$r_\sigma(f(\lambda))=r_\sigma(g(\lambda))$ for all $\lambda\in\mathbb C$,
	and by Vesentini's Theorem \cite[Theorem 6.4.2]{rans} the function
	$$\mathbb C:\lambda\mapsto \log r_{\sigma}(f(\lambda))$$ is subharmonic on $\mathbb C$. Since $b$ is also normal,
	and since $\rho(a,b)=\rho(b^*,a^*)=0$, one may factorize $f(\lambda)=p(\lambda)q(\lambda)$, where
	$p(\lambda)=e^{\lambda i a }e^{-\lambda i b }$ and $q(\lambda)=e^{\lambda i b^* }e^{-\lambda i a^* }$ are entire each with growth characteristics as
	discussed in Section 1, i.e. either the order is strictly less than $1$ or the order equals $1$ in which case the type equals $0$. Whichever case prevails,
	given $\epsilon>0$ arbitrary, there exists $R(\epsilon)>0$ such that for all $r>R(\epsilon)$
	\begin{equation}\label{WMIW}
	\log r_{\sigma}(f(\lambda))\leq
	\log\left(\|e^{\lambda i a }e^{-\lambda i b }\|\,\|e^{\lambda i b^* }e^{-\lambda i a^* }\|\right)\leq
	\epsilon r
	\end{equation}
	whenever $|\lambda|\leq r$. Observe now, using a compactness argument on $|\lambda|\leq R(\epsilon)$, that
	\eqref{AB} is satisfied. If $\zeta$ lies on the imaginary axis, then
	\begin{align*}
	\limsup\limits_{\lambda\rightarrow\zeta}\log r_{\sigma}(f(\lambda))&=
	\limsup\limits_{\lambda\rightarrow\zeta}\log r_{\sigma}(g(\lambda))\\&\leq
	\lim\limits_{\lambda\rightarrow\zeta}\log \left(\|e^{\lambda i (a-a^*)} \|\,\|e^{-\lambda i (b-b^*)}\|\right)\\&
	=\log \left(r_\sigma(e^{\zeta i (a-a^*)})\,r_\sigma(e^{-\zeta i (b-b^*)})\right)=0,
	\end{align*}
	so that \eqref{BDY} holds.
	Finally \eqref{WMIW}, implies that
	\[\limsup\limits_{{t\rightarrow\infty}\atop{t>0}}\frac{\log r_{\sigma}(f(t))}{t}=0.\]
	It thus follows, from Theorem~\ref{PHT}, that
	$r_{\sigma}(f(\lambda))=r_{\sigma}(g(\lambda))\leq1$ for all $\lambda\in H$. Using the same argument with $a$ and $b$ replaced by $-a$ and $-b$ respectively
	we see that $r_{\sigma}(f(\lambda))=r_{\sigma}(g(\lambda))\leq1$ for all $\lambda\in\mathbb C$. Now define an entire function
	
	\begin{displaymath}
	h(\lambda)=\left\{\begin{array}{cc}\left({e^{\lambda i (a-a^*) }e^{-\lambda i (b-b^*) }-1}\right)/{\lambda} & \mbox{ if }\lambda\not=0\\
	i(a-a^*)-i(b-b^*) & \mbox{ if }\lambda=0.\end{array}\right.
	\end{displaymath}
	Since $r_{\sigma}(g(\lambda))$ is bounded on $\mathbb C$ it follows that
	$\limsup_{|\lambda|\rightarrow\infty}r_{\sigma}(h(\lambda))=0$. But $r_{\sigma}(h(\lambda))$ is subharmonic on $\mathbb C$ and therefore, by a version of Liouville's Theorem for subharmonic functions, it must be constantly zero on $\mathbb C$. In particular, we see that $r_\sigma(i(a-a^*)-i(b-b^*))=0$. But  $i(a-a^*)-i(b-b^*)$ being self-adjoint it follows that
	$a-a^*=b-b^*$. Writing $c:=a-a^*=b-b^*$ we see that $c$ commutes with both $a$ and $b$ from which we then obtain
	$$\rho\left((a+a^*)/2,(b+b^*)/2\right)=\rho\left(a-c/2,b-c/{2}\right)=\rho(a,b)=0.$$
	So, using Corollary~\ref{SAE}, we get $a+a^*=b+b^*$ and hence that $a=b$ as advertised.
\end{proof}

Using  P\'{o}lya's result it readily follows that $a$ is self-adjoint if and only if $\rho(a^*,a)=0$ and $\sup_{n\in\mathbb Z}\|e^{ina}\|<+\infty$, and, using Theorem~\ref{spec} (ii),  that $a\in A^{-1}$ is unitary if and only if $\rho(a^*,a^{-1})=0$ and $r_\sigma\left(a^{\pm1}\right)=\left\|a^{\pm1}\right\|$. In both of the aforementioned, neither $\rho(a^*,a)=0$ nor $\rho(a^*,a^{-1})=0$ on their own is sufficient to establish $a$ self-adjoint or unitary respectively. For example, let $V$ be the Volterra operator on $L_2$. Then, since $V$ is quasinilpotent, $\rho(V^*,V)=0$ but of course $V^*\not=V$. In the second instance set $a(t)=e^{tV},\ t\in\mathbb R$ and notice that $\rho\left((a(t))^*,(a(t))^{-1}\right)=0$ for each $t\in\mathbb R$. But if the implication of this is that $a(t)$ is unitary for each real $t$, then it would mean that the Volterra operator is normal which is absurd. On the positive side Theorem~\ref{spec}, which improves on \cite[Lemma 1, Lemma 2]{frunza}, shows that
$\rho(a^*,a)=0$ implies that $a$ is at least ``spectrally self-adjoint", and $\rho(a^*,a^{-1})=0$ implies that $a$ is at least ``spectrally unitary".

\begin{theorem}\label{spec}
	Let $A$ be a $C^*$-algebra and let $a\in A$.
	\begin{itemize}
		\item[(i)]{ If $\rho(a^*,a)=0$ then $\sigma(a)\subset\mathbb R$.  }
		\item[(ii)]{ If $\rho(a^*,a^{-1})=0$ then $\sigma(a)\subseteq\{\lambda\in\mathbb C:|\lambda|=1\}$.  }
	\end{itemize}
\end{theorem}

\begin{proof}
	(i) Of course $\rho(a^*,a)=0\Rightarrow\rho(ia^*,ia)=0$. Suppose
	$\alpha+\beta i\in\sigma(a)$ where $\alpha,\beta\in\mathbb R$ with $\beta\not=0$. The function $h(\lambda)=e^{\lambda i a^*}e^{-\lambda i a}$ has either $\omega_h<1$, or $\omega_h=1$.
	Assume first $\omega_h=1$ so that $\tau_f=0$ as pointed out earlier.  Since $\rho(ia^*,ia)=0$, and since $|\beta|>0$, there exists $R>0$ such that
	$\sup_{|\lambda|\leq r}\|h(\lambda)\|<e^{|\beta|r}$ for $r>R$. But, on the other hand, if $t\in\mathbb R$, then $h(t)$ is self-adjoint and it follows via the Spectral Mapping Theorem that
	\[\|h(t)\|=\|e^{-ita}\|^2\geq r_\sigma(e^{-2ita})\geq|e^{-2it(\alpha+\beta i)}|=e^{2\beta t}.\] This would imply that
	\[e^{2|\beta| r}\leq\sup_{t\in[-r,r]}\|h(t)\|\leq\sup_{|\lambda|\leq r}\|h(\lambda)\|<e^{|\beta|r}\text{ for }r>R,\]
	yielding a contradiction. If $\omega_h<1$ then (now using the definition of order) there exists $0<K<1$ and $R>0$ such that
	\[e^{2|\beta| r}\leq\sup_{t\in[-r,r]}\|h(t)\|\leq\sup_{|\lambda|\leq r}\|h(\lambda)\|<e^{r^K}\text{ for }r>R,\]
	again giving a contradiction. We conclude that $\sigma(a)\subset\mathbb R$.
	
	(ii) Apply \eqref{invol} to get $\rho(a^*,a^{-1})=0\Leftrightarrow\rho((a^{-1})^*,a)=0.$ Then follow the idea in \cite{frunza} to obtain:
	\[\left(C_{(a^{-1})^*,a}^n\mathbf 1\right)a^{-n}=\sum_{j=0}^n{n \choose j}(-1)^j(a^{j-n})^*a^{j-n}\] which, by induction, implies
	\begin{equation}\label{ind}
	(a^{-n})^*a^{-n}=\sum_{j=0}^n{n \choose j}\left(C_{(a^{-1})^*,a}^j\mathbf 1\right)a^{-j}.
	\end{equation}
	Now, let $\epsilon>0$ be arbitrary. Since $\rho((a^{-1})^*,a)=0$ there exists $K(\epsilon)>0$ such that $\|C_{(a^{-1})^*,a}^{j}\mathbf 1\|<\epsilon^{j}$ for $j>K$. Of course we can find $M(\epsilon)>1$ such that $\|C_{(a^{-1})^*,a}^{j}\mathbf 1\|<M\epsilon^{j}$ for $j\leq K$ whence \eqref{ind} gives
	\[\|(a^{-n})^*a^{-n}\|\leq M(\epsilon\|a^{-1}\|+1)^{n}\quad\text{ for }n\in\mathbb N.\]
	But $(a^{-n})^*a^{-n}$ is self-adjoint so that $\|a^{-n}\|^{2/{n}}\leq M^{1/n}(\epsilon\|a^{-1}\|+1)$ from which we obtain $r_\sigma(a^{-1})\leq(\epsilon\|a^{-1}\|+1)^{1/2}$,
	and consequently $r_\sigma(a^{-1})\leq1$. Since $\rho(a^*,a^{-1})=0$ we can similarly prove (cf. \cite{frunza}) that $r_\sigma(a)\leq1$ which establishes the result.
\end{proof}

\bibliographystyle{amsplain}

\begin{thebibliography}{99}
	\bibitem{aup}  B.~Aupetit, \emph{A primer on spectral theory}, Springer-Verlag, New York 1991.
	\bibitem{b+r}  R.~Brits and H.~Raubenheimer, Finite spectra and quasinilpotent equivalence in Banach algebras, \emph{Czechoslovak Math. J.} {\bf 62} (2012), 1101--1116.
	\bibitem{c+f}  I.~Colojoara\v a and C.~Foia\c s, Quasinilpotent equivalence of not necessarily commuting operators, \emph{J. Math. Mech.} {\bf 15} (1966), 521--540.
	\bibitem{duggal}  B.P.~Duggal, Asymptotic intertwining by the identity operator and permanence of spectral properties, \emph{Banach J. Math. Anal.} {\bf 7}(1) (2013), 186--195.
	\bibitem{duggal2}  B.P.~Duggal, I.H. Jeon and I.H. Kim, Upper triangular operator matrices, asymptotic intertwining and Browder, Weyl theorems, \emph{J. Inequal. Appl.} (2013), {2013}:268
	\bibitem{f+v}  C.~Foia\c s and F.-H.~Vasilescu, On the spectral theory of commutators,  \emph{J. Math. Anal. Appl.} {\bf 31} (1970), 473--486.
	\bibitem{frunza} S.~Frunz\v a, Jordan operators on Hilbert space, \emph{J. Operator Theory} {\bf 18} (1987), 201--212.
	\bibitem{l+n}  K.B.~Laursen and M.M.~Neumann, \emph{An introduction to local spectral theory}, Oxford University Press, 2000.
	\bibitem{levin}  B.~Ya.~Levin, \emph{Lectures on entire functions}, AMS 1996.
	\bibitem{levinson}  N.~Levinson,  On a problem of Polya, \emph{Amer. J. Math.} {\bf 58}(4) (1936), 791--798.
	\bibitem{palmer}  T.W.~Palmer, \emph{ Banach algebras and the general theory of $^*$-Algebras, Vol. I algebras and Banach algebras}, Cambridge University Press, 1994.
	\bibitem{rans}  T.~Ransford, \emph{Potential theory in the complex plane}, LMS Student Texts 28, Cambridge University Press, 1995.
	\bibitem{razpet}  M.~Razpet, The quasinilpotent equivalence in Banach algebras, \emph{J. Math. Anal. Appl.} {\bf 166} (1992), 378--385.
	\bibitem{rudin2} W.~Rudin, \emph{Real and complex analysis}, McGraw-Hill, 1987.
	\bibitem{rudin1} W.~Rudin, \emph{Functional analysis}, McGraw-Hill, 1991.
	\bibitem{vas3} F.-H.~Vasilescu, \emph{Analytic functional calculus and spectral decompositions}, Editura Academiei and D. Reidel Publishing Company, 1982.
	\bibitem{vas1} F.-H.~Vasilescu, Some properties of the commutator of two operators, \emph{J. Math. Anal. Appl.} {\bf 23} (1968), 440--446.
	\bibitem{vas2} F.-H.~Vasilescu, Spectral distance of two operators, \emph{Rev. Roumaine Math. Pures Appl.} {\bf 12} (1967), 733--736.
\end{thebibliography}

\end{document}